\documentclass[10pt]{article}
\usepackage{amsmath,amsfonts,amssymb,amsthm,graphicx,stmaryrd}
\usepackage{enumerate}
\usepackage{etoolbox}
\usepackage{color}
\usepackage{algorithm}
\usepackage[noend]{algpseudocode}
\usepackage[colorlinks=true,citecolor=blue,pdfpagemode=UseNone,pdfstartview=FitH]{hyperref}

\allowdisplaybreaks[4]

\newcommand{\p}{p}
\newcommand{\e}{e}

\renewcommand{\d}{\,\mathrm{d}}
\newcommand{\dd}{\mathrm{d}}
           
\newcommand{\E}{\mathbb{E}}
\renewcommand{\P}{\mathbb{P}}
\newcommand{\R}{\mathbb{R}}

\newcommand{\EEE}{\mathcal{E}}

\DeclareMathOperator{\sign}{sign}
\DeclareMathOperator{\Sum}{sum}
\DeclareMathOperator{\KL}{KL}
\newcommand{\from}{\mathbin{\|}}

\theoremstyle{plain}
\newtheorem{theorem}{Theorem}[section]

\newtheorem{lemma}[theorem]{Lemma}
\newtheorem{proposition}[theorem]{Proposition}

\theoremstyle{definition}

\theoremstyle{remark}
\newtheorem{remark}[theorem]{Remark}

\newcommand{\abstr}
  {The notion of an \e-value has been recently proposed
  as a possible alternative to critical regions and \p-values
  in statistical hypothesis testing.
  In this paper we consider testing the nonparametric hypothesis of symmetry,
  introduce analogues for \e-values of three popular nonparametric tests,
  define an analogue for \e-values of Pitman's asymptotic relative efficiency,
  and apply it to the three nonparametric tests.
  We discuss limitations of our simple definition of asymptotic relative efficiency
  and list directions of further research.}

\begin{document}
\title{Nonparametric \e-tests of symmetry}
\author{Vladimir Vovk\thanks%
  {Department of Computer Science,
  Royal Holloway, University of London,
  Egham, Surrey, UK.
  E-mail: \href{mailto:v.vovk@rhul.ac.uk}{v.vovk@rhul.ac.uk}.}
\and Ruodu Wang\thanks%
  {Department of Statistics and Actuarial Science,
  University of Waterloo,
  Waterloo, Ontario, Canada.
  E-mail: \href{mailto:wang@uwaterloo.ca}{wang@uwaterloo.ca}.}}

\maketitle
\begin{abstract}
  \smallskip
  \abstr
\end{abstract}

\section{Introduction}

The study of the efficiency of nonparametric tests
that started in the late 1940s is often regarded as a success story in statistics.
Some nonparametric tests, such as Wilcoxon's signed-rank and rank-sum tests,
are highly efficient even when used in the framework of popular parametric models,
such as the Gaussian model.
Theoretical results mostly concern asymptotic efficiency of those tests,
but there is also empirical evidence for their finite-sample efficiency.
While some nonparametric tests (such as Wilcoxon's)
became very popular after their high efficiency had been discovered,
others (such as Wald and Wolfowitz's run test)
were gradually discarded from the statistical literature
after their low efficiency had been demonstrated
\cite[Introduction]{Nikitin:1995}.

The usual approach to hypothesis testing is based on critical regions or \p-values,
but in this paper we replace them with their alternative, \e-values
(see, e.g., \cite{Vovk/Wang:2021,Shafer:2021,Grunwald/etal:arXiv1906}).
We show that some of the old results about the efficiency of nonparametric tests
carry over to hypothesis testing based on \e-values.
To distinguish our notions of power, tests, etc.,
from the standard notions, we add the prefix ``\e-''.
(The prefix ``\p-'' is sometimes added to signify standard notions based on \p-values,
but in this paper we rarely need it since the key notion that we are interested in,
Pitman's asymptotic relative efficiency,
is defined in terms of critical regions rather than \p-values.)

We explain basics of \e-testing in Sect.~\ref{sec:e-testing},
and in particular, we state an analogue of the Neyman--Pearson lemma in \e-testing.
In the following section, Sect.~\ref{sec:parametric},
we give a simple example of a parametric \e-test,
one for testing the null hypothesis $N(0,1)$ against an alternative $N(\theta,1)$
in an IID situation.

In Sect.~\ref{sec:Fisher} we give the first,
and in some sense most powerful,
of the three examples of nonparametric \e-tests that we discuss in this paper.
It was introduced by Fisher in his 1935 book \cite{Fisher:1935local}.
Our nonparametric null hypothesis is that of symmetry around 0
(and for simplicity we consider independent observations
coming from a continuous distribution).

The material of Sects.~\ref{sec:e-testing}--\ref{sec:Fisher} is standard.
After that (Sect.~\ref{sec:power})
we define the asymptotic relative efficiency of \e-tests
in the spirit of Pitman's definition \cite{Pitman:1948}.
We regard our definition of asymptotic relative efficiency
as a direct translation of the classical definition.
Then in Sect.~\ref{sec:ARE}
we compute the Pitman-type asymptotic relative efficiency
of the Fisher-type test discussed in Sect.~\ref{sec:Fisher}.
This is complemented by similar computations
for \e-versions of the sign test in Sect.~\ref{sec:sign}
and Wilcoxon's signed-rank test in Sect.~\ref{sec:Wilcoxon}.
Our results for all three tests agree perfectly with the classical results.
This is just a first step, and in Sect.~\ref{sec:conclusion}
we discuss limitations of our approach (which are considerable)
and list natural directions of further research.

\section{General principles of \e-testing}
\label{sec:e-testing}

Let $P$ be a given probability measure on a sample space $\Omega$
(a measurable space).
Our \emph{null hypothesis} is $\{P\}$;
it is simple in the sense of containing a single probability measure.
(We will sometimes also refer to $P$ as our null hypothesis.)

We observe $\omega\in\Omega$ and are interested in whether $\omega$ was generated from $P$.
An \emph{\e-variable} for testing $P$ is an $[0,\infty]$-valued random variable $E$
such that $\int E\d P\le1$.
In order to be used for testing, we need to choose $E$ before we observe $\omega$.
By Markov's inequality, $E$ can be large only with a small probability
(for any threshold $c>1$, $P(E\ge c)\le1/c$);
therefore, observing a large $E$ casts doubt on $\omega$ being generated from $P$.

In the classical Neyman--Pearson approach to hypothesis testing,
in addition to $P$ we also have an alternative hypothesis $Q$.
The \emph{\e-power} of an \e-variable $E$ is then defined
as $\int \log E \d Q$.
This is an analogue of the usual notion of power,
but it only works in regular cases.
One of such regular cases will be discussed in the next section.
The following lemma is very well known
(see, e.g., \cite[Sect.~2.2.1]{Shafer:2021} and the references therein), and we provide a simple proof.

\begin{lemma}\label{lem:power}
  For given null and alternative hypotheses $P$ and $Q$, respectively,
  such that $Q\ll P$,
  the largest \e-power is attained by the likelihood ratio $\dd Q/\dd P$:
  for any \e-variable~$E$,
  \begin{equation}\label{eq:to-prove}
    \int \log E\d Q
    \le
    \int \log\frac{\dd Q}{\dd P}\d Q.
  \end{equation}
  And if $Q\ll P$ is violated,
  the largest \e-power is $\infty$.
\end{lemma}

The likelihood ratio $\dd Q/\dd P$ in Lemma~\ref{lem:power}
is understood to be the Radon--Nikodym derivative of $Q$ w.r.\ to $P$.

\begin{proof}[Proof of Lemma~\ref{lem:power}]
  If $Q\ll P$ is violated,
  there is an event $A\subseteq\Omega$ such that $P(A)=0$ and $Q(A)>0$.
  Then the \e-power of the \e-variable
  \[
    E(\omega)
    :=
    \begin{cases}
      \infty & \text{if $\omega\in A$}\\
      1 & \text{otherwise}
    \end{cases}
  \]
  is $\infty$.

  It remains to consider the case $Q\ll P$.
  In this case, let $q$ be a probability density function of $Q$ w.r.\ to $P$.
  In terms of $q$, we can rewrite \eqref{eq:to-prove} as
  \begin{equation*}
    \int q \log E \d P
    \le
    \int q \log q \d P,
    \quad
    \text{i.e.,}
    \quad
    \int q \log\frac{E}{q} \d P
    \le
    0.
  \end{equation*}
  The last inequality follows from $\log x \le x-1$.
\end{proof}

According to Lemma~\ref{lem:power},
which is an analogue for \e-values of the Neyman--Pearson lemma,
the optimal \e-variable for testing a null hypothesis $P$
against an alternative $Q\ll P$ is the likelihood ratio $\dd Q/\dd P$.
The maximum \e-power is
\[
  \KL(Q\from P)
  :=
  \int
  \log\frac{\dd Q}{\dd P}
  \d Q
\]
(cf.\ \cite[Sect.~2.3]{Shafer:2021} and \cite[Theorem 1]{Grunwald/etal:arXiv1906}).
This is simply the Kullback--Leibler divergence \cite{Kullback/Leibler:1951}
of the alternative $Q$ from the null hypothesis $P$;
we will call it the \emph{optimal \e-power}.

We will sometimes refer to $\log E$ as the \emph{observed \e-power} of $E$;
the \e-power is then the expectation of the observed \e-power
w.r.\ to the alternative hypothesis~$Q$.

The notion of \e-power is very close to Shafer's \cite{Shafer:2021} implied target,
the main difference being that the implied target only depends
on the null hypothesis $P$ and the \e-variable~$E$.

As a short detour,
let us check that our notion of \e-power enjoys a natural property
in testing with multiple \e-values.
Denote by $\Pi^Q$ the function
\begin{equation}\label{eq:Pi}
  \Pi^Q:
  E\mapsto\int \log E \d Q
\end{equation}
that maps an \e-variable to its \e-power.
Independent \e-variables $E_1,\dots,E_K$
can be combined into one \e-variable using a merging function,
the most common choices being convex mixtures of the product functions
\[
  F_M:(e_1,\dots,e_K)\mapsto\prod_{k\in M} e_k,
\]
where $M$ is a subset of $\{1,\dots,K\}$,
with $F_{\varnothing}$ set to $1$.
Denote by $\mathcal{M}$ the convex hull of all functions $F_M$.
Useful elements of the class $\mathcal{M}$ are U-statistics,
symmetric merging functions discussed in \cite[Sect.~4]{Vovk/Wang:2021}.

\begin{proposition}\label{prop:simple}
  Let $\mathbf{E}=(E_1,\dots,E_K)$ be a vector of independent \e-variables.
  \begin{enumerate}
  \item[(i)]
    For all $F\in\mathcal{M}$,
    $F(\mathbf{E})$ is an \e-variable.
  \item[(ii)]
    If $\Pi^Q(E_k)>0$ for each $k=1,\dots,K$,
    then $\Pi^Q(F(\mathbf{E}))>0$ for all $F\in\mathcal{M}\setminus\{F_{\varnothing}\}$.
  \item[(iii)]
    If $\Pi^Q(E_k)\ge0$ for each $k=1,\dots,K$,
    then $\Pi^Q(F(\mathbf{E}))\ge0$ for all $F\in\mathcal{M}$.
  \end{enumerate}
\end{proposition}
\begin{proof}
  Part~(i) follows from the fact that the product of independent \e-variables is an \e-variable,
  and a convex mixture of \e-variables is an \e-variable.
  Next we prove~(ii). 
  For all $M$ other than $M=\varnothing$,
  we have
  \[
    \Pi^Q(F_M(\mathbf{E}))
    =
    \sum_{k\in M}
    \Pi^Q(E_k)
    >
    0,
  \]
  and $\Pi^Q(F_{\varnothing}(\mathbf{E}))=0$.
  Note that the mapping \eqref{eq:Pi}
  is concave on the set of nonnegative random variables.
  Since $F(\mathbf{E})$ is a convex mixture of $F_M(\mathbf{E})$ for $M\subseteq\{1,\dots,K\}$,
  we get $\Pi^Q(F(\mathbf{E}))\ge 0$,
  and the inequality is strict unless $F=F_{\varnothing}$.
  This proves~(ii).
  The case~(iii) is similar to~(ii). 
\end{proof}

Proposition \ref{prop:simple} shows that \e-power remains positive
when combining independent \e-values with positive \e-power
using a large class of merging functions.
As a special case of Proposition~\ref{prop:simple} applied to only one \e-variable,
if $\Pi^Q(E)>0$,
then $\Pi^Q(1-\lambda +\lambda E)>0$ for all $\lambda\in(0,1]$.
The operation of changing $E$ to $1-\lambda+\lambda E$ is common in building e-processes;
see, e.g., \cite{Waudby-Smith/Ramdas:2022}.

\section{A parametric \e-test}
\label{sec:parametric}

We start our discussion of specific \e-tests from a very simple parametric case,
that of the Gaussian statistical model $Q_\theta:=N(\theta,1)$, $\theta\in\R$,
with the variance known to be 1.
We observe realizations of independent $Z_1,\dots,Z_n\sim N(\theta,1)$.
The null hypothesis $P$ is $N(0,1)$,
and we are interested in the alternatives
$Q=Q_{\theta}=N(\theta,1)$ for $\theta\ne0$.

For observations $z_1,\dots,z_n$ and a given alternative $N(\theta,1)$,
the likelihood ratio of the alternative to the null hypothesis is
\begin{equation}\label{eq:LR}
  E_{\theta}(z_1,\dots,z_n)
  :=
  \frac
  {
    \exp
    \left(
      -\frac12
      \sum_{i=1}^n
      (z_i - \theta)^2
    \right)
  }
  {
    \exp
    \left(
      -\frac12
      \sum_{i=1}^n
      z_i^2
    \right)
  }
  =
  \exp
  \left(
    \theta
    \sum_{i=1}^n
    z_i
    -
    \frac12
    n
    \theta^2
  \right).
\end{equation}
The corresponding optimal \e-power is
\begin{equation}\label{eq:optimal}
  \int \log E_{\theta} \d Q_{\theta}
  =
  \theta
  n
  \theta
  -
  \frac12
  n
  \theta^2
  =
  \frac12
  n \theta^2.
\end{equation}

The interpretation of the optimal \e-power \eqref{eq:optimal}
usually depends on the law of large numbers and its refinements
(such as the central limit theorem and large deviation inequalities).
The presence of $\log$ in the definition $\int \log E \d Q$ of the \e-power of $E$ under the alternative $Q$
reflects the fact that a typical \e-value is obtained by multiplying components
coming from the individual observations $z_i$.
This can be seen from \eqref{eq:LR}
(and also expressions \eqref{eq:E-Fisher}, \eqref{eq:equivalent}, and \eqref{eq:E-Wilcoxon-full} below,
which are typical).
Taking the logarithm leads to a much more regular distribution,
which is, e.g., approximately Gaussian under standard regularity conditions.
In the case of \eqref{eq:LR},
the key component of the logarithm is $\sum_{i=1}^n z_i$,
and we can apply, e.g., the central limit theorem
to see that the observed \e-power is between the narrow limits
$\frac12 n \theta^2 \pm c \sqrt{n} \theta$
with probability close (in this particular case, even exactly equal) to $\Phi(c)-\Phi(-c)$,
where $c>0$ and $\Phi$ is the standard Gaussian cumulative distribution function.

\begin{remark}\label{rem:full}
  To get the full idea of the power of $E$ under $Q$,
  we need the whole distribution of the observed \e-power $\log E$ under $Q$,
  and replacing it by its expectation is a crude step.
  (The next step might be, e.g.,
  complementing the expectation with the standard deviation of $\log E$ under $Q$.)
  We leave such more realistic notions of power for future research.
\end{remark}

We regard the family \eqref{eq:LR} of \e-variables
as a test (an \emph{\e-test}) of the null hypothesis $N(0,1)$.
While for several important statistical models
there are uniformly most powerful \p-tests
(see, e.g., \cite[Chap.~3]{Lehmann/Romano:2022}),
this is not the case for \e-tests,
and the \e-tests considered in this paper are always families of \e-variables.

The fact that the \e-variable \eqref{eq:LR} depends on the unknown alternative parameter $\theta$
is a disadvantage.
A natural way out is to integrate it under the prior distribution $N(0,1)$ over $\theta$,
which gives us the \e-variable
\begin{multline}\label{eq:Gaussian-mix}
  \frac{1}{\sqrt{2\pi}}
  \int
  \exp
  \left(
    \theta
    \sum_{i=1}^n
    z_i
    -
    \frac12
    n
    \theta^2
    -
    \frac12
    \theta^2
  \right)
  \dd\theta\\
  =
  \sqrt{\frac{1}{n+1}}
  \exp
  \left(
    \frac{1}{2n+2}
    \left(
      \sum_{i=1}^n
      z_i
    \right)^2
  \right)
\end{multline}
(cf.\ Remark~\ref{rem:formula} below).
Notice that the operation of integration makes the \e-variable ``two-sided'':
while \eqref{eq:LR} is monotone in $\sum_i z_i$,
\eqref{eq:Gaussian-mix} is monotone in $\left|\sum_i z_i\right|$.
The remaining disadvantage of the \e-variable \eqref{eq:Gaussian-mix}
is that it is valid only under the simple Gaussian null hypothesis $N(0,1)$.
In the following sections we will replace this simple null hypothesis
with a composite nonparametric one.

\begin{remark}\label{rem:formula}
  In our computations in this paper we often use the formula
  \begin{equation*}
    \int
    \exp
    \left(
      -A x^2
      +
      B x
    \right)
    \dd x
    =
    \sqrt{\frac{\pi}{A}}
    \exp
    \left(
      \frac{B^2}{4A}
    \right),
  \end{equation*}
  where $A>0$ and $B\in\R$.
 Equations~\eqref{eq:LR} and~\eqref{eq:Gaussian-mix} are simple calculations,
 and they appear in the context of mixture martingales,
 which date back to, at least, the work of Robbins
 (e.g., \cite{Robbins:1970});
 see also the more recent \cite{Howard/etal:2021} and the references therein.
\end{remark}

\section{Fisher-type nonparametric \e-test of symmetry}
\label{sec:Fisher}

Let $Z_1,\dots,Z_n$ be continuous IID random variables.
We are interested in the null hypothesis that their distribution is symmetric around 0.
This is an example of a nonparametric hypothesis,
since the distribution of $Z_1,\dots,Z_n$ is not described in a natural way
by finitely many real-valued parameters.
Intuitively, we are interested in two alternatives:
the one-sided alternative that $Z_i$, even though IID, are not symmetric but shifted to the right;
and the two-sided alternative that $Z_i$ are shifted to the right or to the left.

A typical case in applications is where $Z_i:=Y_i-X_i$,
$X_i$ is a pre-treatment measurement, and $Y_i$ is a post-treatment measurement,
and we are interested in whether the treatment has any effect.
Assuming that raising $X_i$ is desirable,
the one-sided alternative is that the treatment is beneficial.

We will formalize our null hypothesis
in a way similar to repetitive and one-off structures
\cite[Sects.~11.2.4 and~11.2.5]{Vovk/etal:2022book}.
However, we will not need general definitions
and will adapt them to our special case.

The \emph{symmetry model} for a sample size $n$
is the pair $(t,b)$,
where $t:\R^n\to\Sigma$ is the mapping
\[
  t:
  (z_1,\dots,z_n)
  \mapsto
  \left(
    \left|z_1\right|,
    \dots,
    \left|z_n\right|
  \right)
\]
from the \emph{sample space} $\R^n$ to the \emph{summary space} $[0,\infty)^n$,
and $b$ is the Markov kernel that maps each summary $(z_1,\dots,z_n)\in[0,\infty)^n$
to the uniform probability measure on the set
\begin{equation}\label{eq:b}
  t^{-1}(z_1,\dots,z_n)
  =
  \left\{
    (j_1 z_1,\dots,j_n z_n)
    \mid
    (j_1,\dots,j_n)\in\{-1,1\}^n
  \right\}.
\end{equation}
An \emph{\e-variable} for testing the null hypothesis of symmetry
is a function $E:\R^n\to[0,\infty]$
such that $\int E \d b(t(z_1,\dots,z_n))\le1$ for all $z_1,\dots,z_n$.
It is \emph{admissible} if $\le$ holds as $=$ for all $z_1,\dots,z_n$;
in other words, if it ceases to be an \e-variable (w.r.\ to the symmetry model)
as soon as its value is increased at any point.

\begin{remark}
  The definition of admissibility that we give
  is adapted to our current context;
  see \cite[Sect.~9]{Ramdas/etal:arXiv2009} for a more general discussion.
\end{remark}

In this section we define the first of our three \e-tests for testing symmetry.
We are interested in the \e-variables of the form
\begin{equation}\label{eq:E-exp}
  E_{\lambda}(z_1,\dots,z_n)
  :=
  \exp
  \left(
    \lambda S(z_1,\dots,z_n) - C
  \right),
\end{equation}
where $S(z_1,\dots,z_n) := \sum_{i=1}^n z_i$,
$\lambda>0$ is a positive parameter,
and $C$ is chosen to make $E$ an admissible \e-variable,
i.e.,
\[
  C
  =
  C(\lambda,t(z_1,\dots,z_n))
  :=
  \log
  \int
  \exp(\lambda S)
  \dd b(t(z_1,\dots,z_n))
\]
(in other words, $C:=\log\E\exp(\lambda S)$,
the expectation being under the null hypothesis,
i.e., under the symmetry model).
Lemma~\ref{lem:C-Fisher} will give a convenient formula for computing $C$.

The form \eqref{eq:E-exp} for our \e-variables
can be justified by the analogy with the \e-variable \eqref{eq:LR}
that we obtained in the Gaussian case.
The expression for the normalizing constant $C$ will, however, be different
and will be derived momentarily.

The justification of the symmetry model
from the point of view of standard statistical modelling
is that, under the null hypothesis of symmetry,
$t$ is a sufficient statistic
giving rise to $b$ as conditional distribution.

For simplicity, we will assume that $z_1,\dots,z_n$ are all different
(under our assumption that the random variables $Z_1,\dots,Z_n$ are continuous,
the realizations will be all different almost surely).

\begin{lemma}\label{lem:C-Fisher}
  The value of $C$ in \eqref{eq:E-exp} is given by
  \begin{equation}\label{eq:C-Fisher}
    C
    =
    \sum_{i=1}^n
    \log
    \frac{e^{\lambda z_i}+e^{-\lambda z_i}}{2}.
  \end{equation}
\end{lemma}

\begin{proof}
  We find
  \begin{equation*}
    e^C
    =
    2^{-n}
    \sum_{j_1=0}^1
    \dots
    \sum_{j_n=0}^1
    e^{\lambda j_1 z_1 + \dots + \lambda j_n z_n}
    =
    2^{-n}
    \prod_{i=1}^n
    \left(
      e^{\lambda z_i}
      +
      e^{-\lambda z_i}
    \right).
  \end{equation*}
  (Alternatively,
  we can see straight away that the average of \eqref{eq:E-Fisher} below
  w.r.\ to $b(t(z_1,\dots,z_n))$ is 1.)
\end{proof}

Plugging \eqref{eq:C-Fisher} into \eqref{eq:E-exp}
gives the \e-variable
\begin{equation}\label{eq:E-Fisher}
  E_{\lambda}(z_1,\dots,z_n)
  =
  e^{-C}
  \prod_{i=1}^n
  e^{\lambda z_i}
  =
  \prod_{i=1}^n
  \frac
  {e^{\lambda z_i}}
  {
    \frac12
    \left(
      e^{\lambda z_i}
      +
      e^{-\lambda z_i}
    \right)
  }.
\end{equation}
This is an \e-version of Fisher's permutation test,
which he introduced and applied to Charles Darwin's data \cite[Chap.~1]{Darwin:1876}
in his 1935 book \cite[Sects.~21 and~21.1]{Fisher:1935local} on experimental design.

Again, since there is no uniformly most powerful \e-test,
we consider a family of \e-variables.
The \e-variable \eqref{eq:E-Fisher} is, of course, admissible.

\begin{figure*}[htbp]
\begin{center}
  \includegraphics[width=10cm]{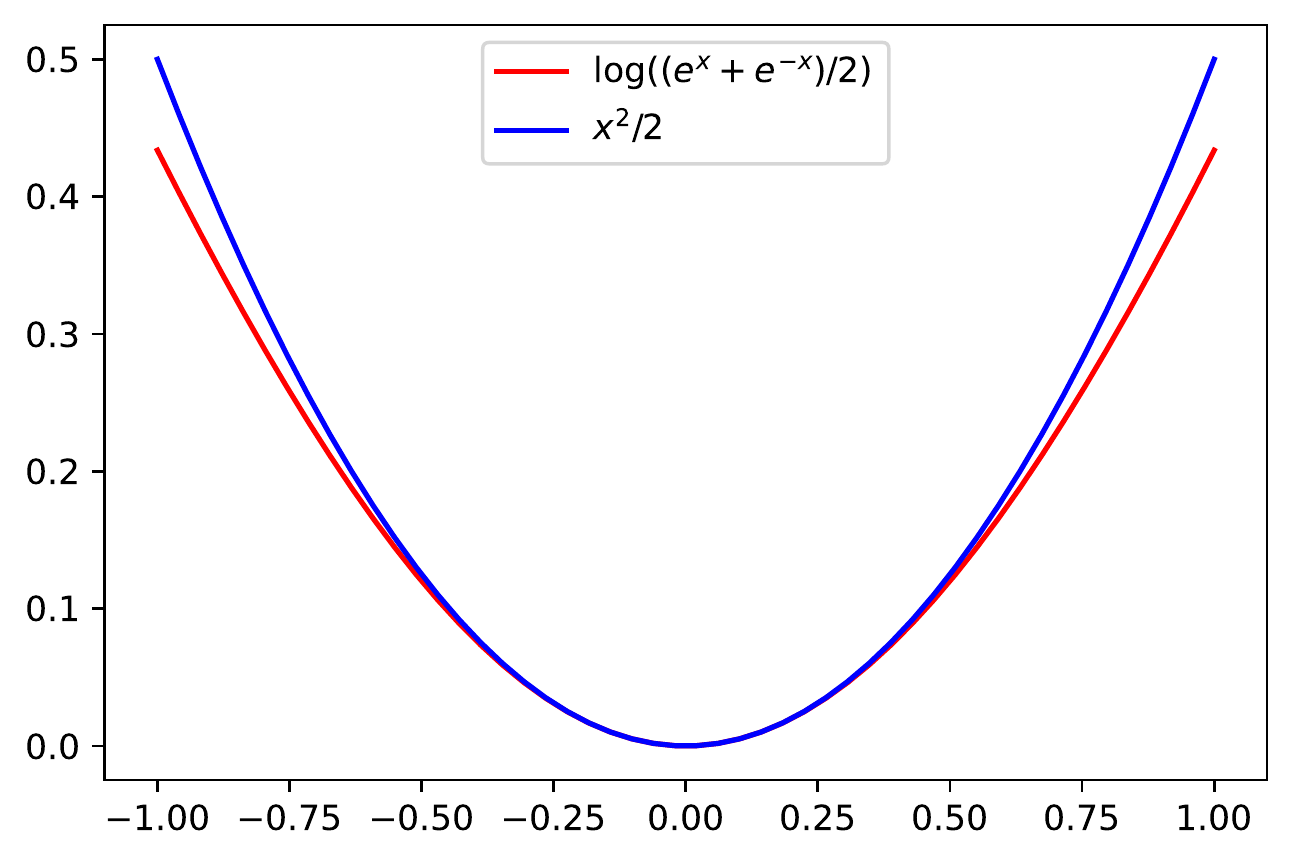}
\end{center} 
\caption{The inequality \eqref{eq:inequality} on the log scale}\label{fig:inequality}
\end{figure*}

The \e-variable \eqref{eq:E-Fisher} dominates
\begin{equation}\label{eq:de_la_Pena}
  E'_{\lambda}(z_1,\dots,z_n)
  :=
  \prod_{i=1}^n
  e^{\lambda z_i - \lambda^2 z_i^2/2},
\end{equation}
in the sense $E'\le E$.
Therefore, $E'$ is also an \e-variable, albeit inadmissible in general.
To check the inequality $E'\le E$,
it suffices to check that
\begin{equation}\label{eq:inequality}
  \frac12
  \left(
    e^x
    +
    e^{-x}
  \right)
  \le
  e^{x^2/2}.
\end{equation}
Expanding both sides into Taylor's series shows that this inequality indeed holds for all $x$.
The inequality is not excessively loose, especially for small values of $x$
(which will be the case that we will be interested in when computing the Pitman efficiencies):
cf.\ Figure \ref{fig:inequality}.

\begin{remark}
  The fact that \eqref{eq:de_la_Pena} is an \e-variable
  was established by de la Pe\~na \cite[Lemma 6.1]{deLaPena:1999}.
  Ramdas et al.\ \cite[Sect.~10]{Ramdas/etal:arXiv2009}
  point out that it is inadmissible,
  and they define several natural admissible alternatives
  to \eqref{eq:E-Fisher}.
  Investigating the asymptotic relative efficiency of those admissible alternatives
  is an interesting direction of further research.
\end{remark}

In order to get rid of the dependence of \eqref{eq:E-Fisher} or \eqref{eq:de_la_Pena}
on $\lambda$,
we can integrate these expression over a prior distribution on $\lambda$.
This can be easily done explicitly
(see Remark~\ref{rem:formula}) in the case of \eqref{eq:de_la_Pena}
and the prior distribution $N(0,1)$ on $\lambda$:
\begin{equation}\label{eq:Fisher-mix}
  \frac{1}{\sqrt{2\pi}}
  \int
  \prod_{i=1}^n
  e^{\lambda z_i - \lambda^2 z_i^2/2 - \lambda^2/2}
  \d\lambda
  =
  \sqrt{\frac{1}{1+\sum_{i=1}^n z_i^2}}
  \exp
  \left(
    \frac{\left(\sum_{i=1}^n z_i\right)^2}{2+2\sum_{i=1}^n z_i^2}
  \right).
\end{equation}

The right-hand side of \eqref{eq:Fisher-mix} is close to the right-hand side of \eqref{eq:Gaussian-mix}
under $N(0,1)$ as the null hypothesis:
this follows from $\sum_{i=1}^n z_i^2 \approx n$
(for large $n$ and with high probability).
However (as noticed in \cite{deLaPena:1999}), this relatively small change
drastically changes the property of validity of the \e-test:
while the right-hand side of \eqref{eq:Gaussian-mix} is an \e-test of $N(0,1)$ only,
the right-hand side of \eqref{eq:Fisher-mix}
is an \e-test of the nonparametric hypothesis of symmetry.

\subsection*{Results for Charles Darwin's data}

In this subsection we will compute Fisher-type nonparametric \e-values
for data used by Darwin \cite[Chap.~1]{Darwin:1876}
to test whether cross-fertilization of plants was advantageous to the progeny
as compared with self-fertilization.
This was an important question from the evolutionary point of view,
and Darwin's preliminary work had convinced him
that cross-fertilization was indeed advantageous;
in particular, nature went to great lengths to prevent self-fertilization
\cite{Darwin:1862}.

\begin{table}
  \begin{center}
  \begin{tabular}{ccc}
    49 & 23 & 56 \\
    $-67$ & 28 & 24 \\
    8 & 41 & 75 \\
    16 & 14 & 60 \\
    6 & 29 & $-48$
  \end{tabular}
  \end{center}
  \caption{Differences in eighths of an inch
    between cross- and self-fertilised plants of the same pair
    (Table~3 in \cite[Sect.~17]{Fisher:1935local})}\label{tab:Darwin}
\end{table}

Table~\ref{tab:Darwin} reports results for a small subset of Darwin's data,
those for maize.
This subset was analyzed for Darwin by Francis Galton
(as Darwin describes in detail in \cite[Chap.~1]{Darwin:1876})
and was reanalyzed by Fisher in \cite[Chap.~3]{Fisher:1935local}.
Fisher offered both parametric analysis (assuming the Gaussian distribution)
and novel nonparametric analysis,
and his finding was that Student's t-test and Fisher's nonparametric test
produce remarkably similar results.

Table~\ref{tab:Darwin} lists the differences in height between 15 pairs of matched plants,
with a cross- and self-fertilized plant in each pair
(meaning a plant grown from a cross- or self-fertilized seed, respectively).
A positive difference means that the cross-fertilized plant is taller,
which we \emph{a priori} expect to happen more often.
Fisher was interested in two alternatives to the null hypothesis of symmetry:
the one-sided alternative of positive observations being more common than negative ones
and the two-sided alternative of asymmetry
(with positive observations being either more or less common than negative ones).

Fisher's \p-value for testing the one-sided hypothesis is 2.634\%,
and his \p-value for testing the two-sided hypothesis is twice as large, 5.267\%.
Therefore, the one-sided \p-value is significant but not highly significant,
whereas the two-sided \p-value is not even significant.

\begin{figure*}[htbp]
\begin{center}
  \includegraphics[width=10cm]{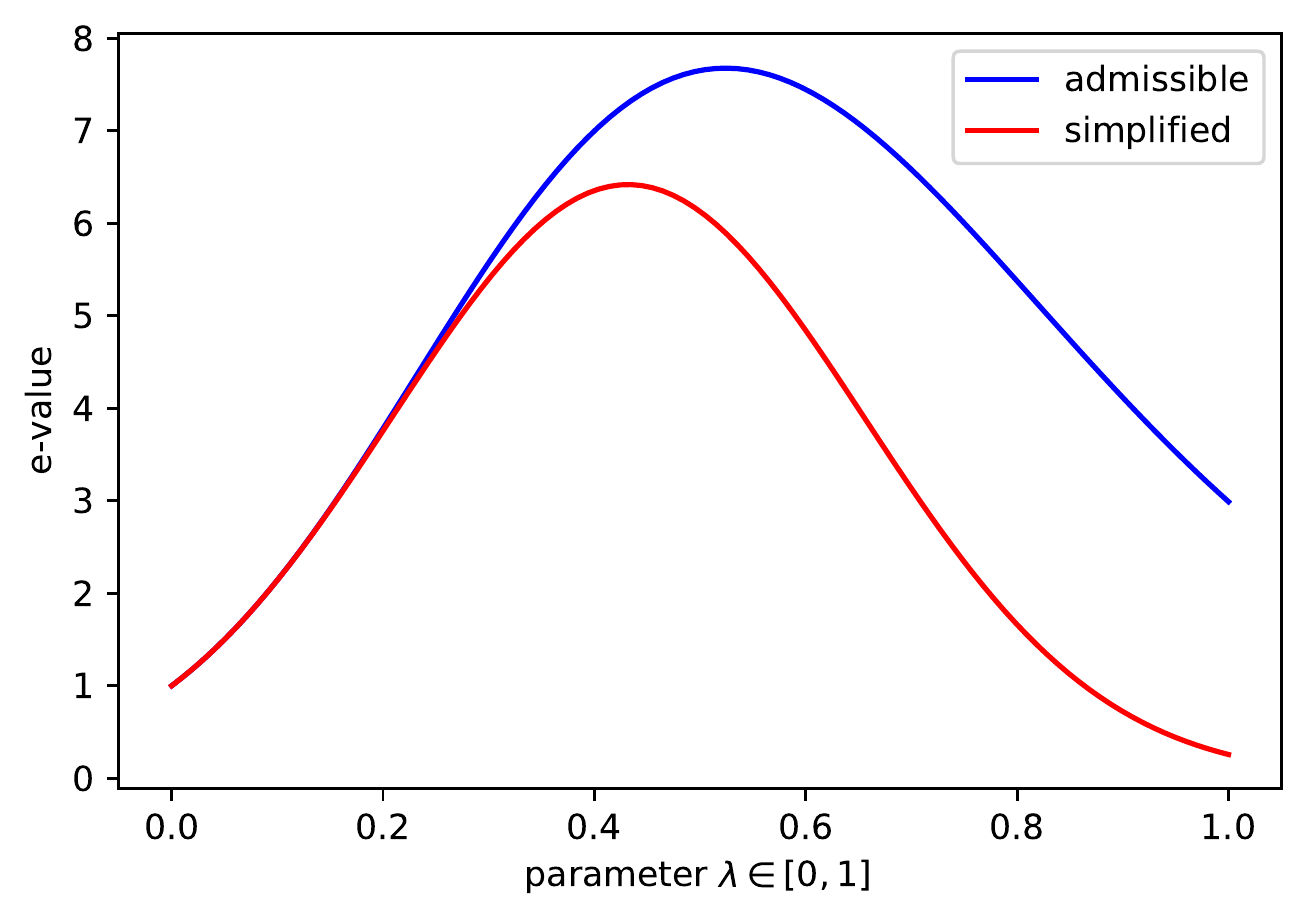}
\end{center} 
\caption{Results for the Fisher-type \e-test applied to Darwin's data}\label{fig:Darwin}
\end{figure*}

Figure~\ref{fig:Darwin} plots the Fisher-type admissible \e-values \eqref{eq:E-Fisher}
(in blue)
and the simplified \e-values \eqref{eq:de_la_Pena}
(in red)
for the parameter $\lambda$ in the range $[0,1]$.
The meaning of $\lambda$ depends on the scale of the numbers $z_1,\dots,z_{15}$
in Table~\ref{tab:Darwin},
and in order to make $\lambda$ less arbitrary we normalize $z_1,\dots,z_{15}$
by dividing them by the standard deviation of these 15 numbers.
Jeffreys's \cite[Appendix~B]{Jeffreys:1961} rule of thumb
is to consider an \e-value of 10 as being analogous to a \p-value of $1\%$
and to consider an \e-value of $\sqrt{10}\approx3.162$
as being analogous to a \p-value of $5\%$.
(See \cite[Sect.~2]{Vovk/Wang:2021} for a more detailed discussion
of relations between \e-values and \p-values.)
This makes Figure~\ref{fig:Darwin} roughly comparable to Fisher's \p-values,
especially if we ignore the inadmissible simplified \e-values.
If we guess in advance that $\lambda:=0.5$ is a good parameter value,
we will get an \e-value of $7.651$.
More realistically, averaging the \e-values for $\lambda\in[0,1]$
will give the one-sided \e-value $5.149$.
Replacing $\lambda\in[0,1]$ by $\lambda\in[-1,1]$
gives the two-sided \e-value $2.633$ not reaching the threshold of $\sqrt{10}$.

\section{Pitman-type asymptotic relative efficiency}
\label{sec:power}

The following definition is in the spirit of Pitman's definition,
which can be found in, e.g., \cite[Sect.~14.3]{vanderVaart:1998}.
Let $(Q_{\theta}\mid\theta\in\Theta)$ be a statistical model,
i.e., a set of probability measures on the real line $\R$,
with the observations generated from one of those probability measures in the IID fashion.
We assume, for simplicity, that $\Theta=\R$ and regard $Q_0$ as the null hypothesis;
informally, the alternative is either one-sided, $\theta>0$,
or two-sided, $\theta\ne0$
(for specific \e-tests,
we will have the same results for one-sided and two-sided Pitman efficiency).
By an \e-variable we mean an \e-variable w.r.\ to $Q_0^n$.
In our asymptotic framework we consider sequences of parameter values $\theta_{\nu}$
that depend on the ``difficulty'' $\nu=1,2,\dots$ of our testing problem;
in the one-sided case we will assume $\theta_{\nu}\downarrow0$
(the sequence is strictly decreasing and converges to 0),
and in the two-sided case we will assume $\theta_{\nu}\to0$.

Let $\EEE^n_1$ and $\EEE^n_2$ be families of \e-variables on $\R^n$;
we are interested in the case where $\EEE^n_1$ is a family of interest to us
(a nonparametric \e-test such as \eqref{eq:E-Fisher} above,
or \eqref{eq:E-sign-lambda} or \eqref{eq:E-Wilcoxon} below)
and $\EEE^n_2$ is the baseline family of all \e-variables on $\R^n$.
The \emph{asymptotic relative efficiency} of $\EEE^n_1$ w.r.\ to $\EEE^n_2$ is $c$
if, for any $\beta>0$
and any $\theta_\nu\downarrow0$ (one-sided case) or $\theta_\nu\to0$ (two-sided case),
we have $n_{\nu,2}/n_{\nu,1}\to c$,
where $n_{\nu,j}$, $j=1,2$, is the minimal number of observations $n$ such that
\[
  \exists E\in\EEE^n_j:
  \int \log E \d Q^n_{\theta_\nu} \ge \beta.
\]
For example, if the asymptotic relative efficiency is 0.5,
the best \e-test in $(\EEE^n_1)$ requires twice as many observations $n$
as the best test in $(\EEE^n_2)$
to achieve the same \e-power (if the best \e-tests exist).

\begin{figure*}[htbp]
\begin{center}
  \includegraphics[width=12cm, trim={8cm 18.5cm 0 1.3cm}, clip]{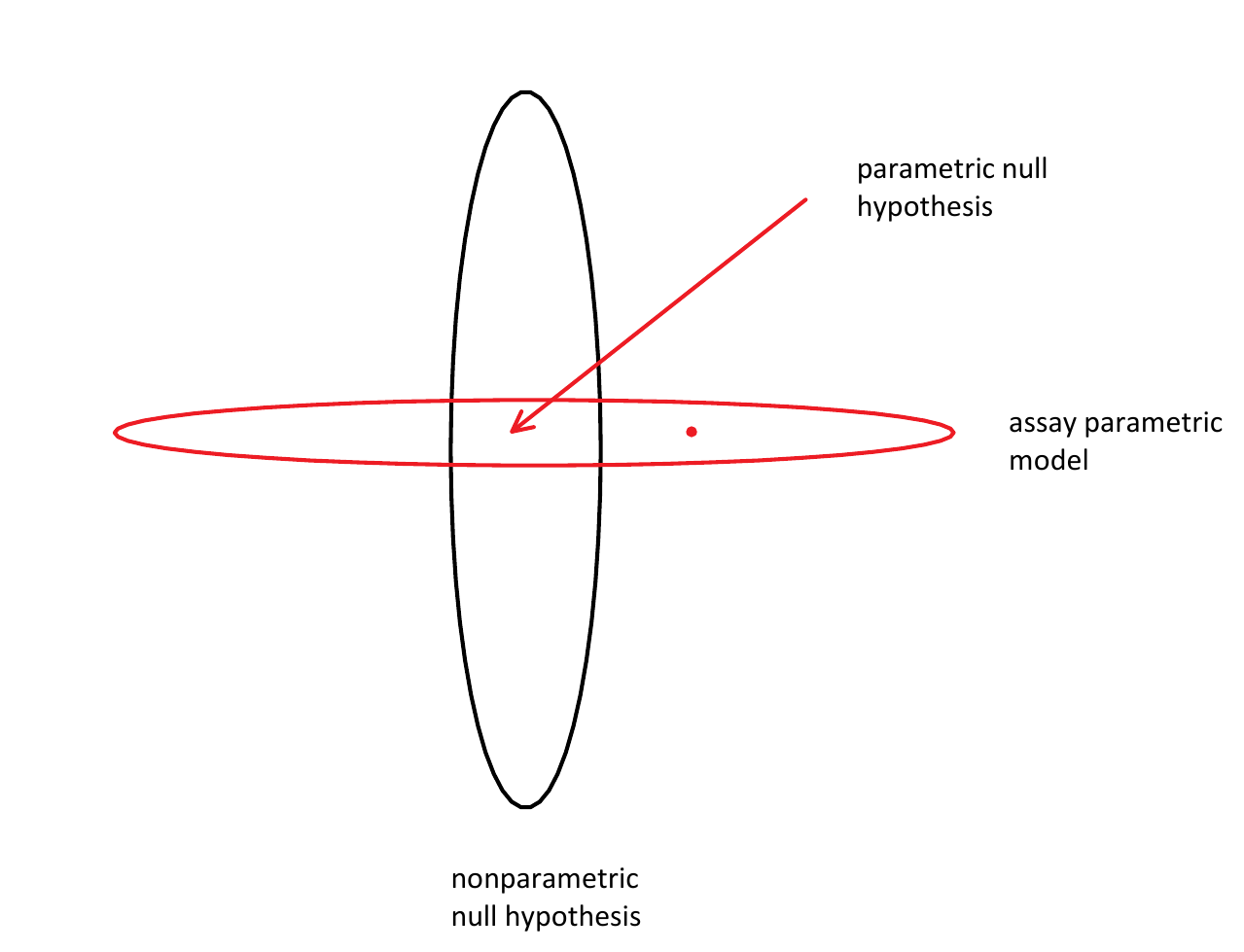}
\end{center} 
\caption{Assaying a non-parametric \e-test}\label{fig:ARE}
\end{figure*}

The idea of using an auxiliary parametric statistical model $(Q_{\theta})$,
such as the Gaussian model, to assay the efficiency of nonparametric \e-tests
is illustrated in Figure~\ref{fig:ARE}.
We are testing a nonparametric null hypothesis
(the hypothesis of symmetry in this paper),
but we are afraid that for a popular parametric model
(the Gaussian model $Q_{\theta}:=N(\theta,1)$ in this paper,
which plays the role of an \emph{assay statistical model})
our testing method loses a lot.
We are interested in the case where the intersection
between the nonparametric null hypothesis and the assay model
contains only one probability measure;
we refer to this intersection
as the \emph{parametric null hypothesis} in Figure~\ref{fig:ARE}
(in this paper, it is $\{N(0,1)\}$).
For a given simple alternative hypothesis $Q=Q_{\theta}$ in the assay model
(shown as the red dot in Figure~\ref{fig:ARE}),
we are hoping to show that the best \e-power
achieved for testing the simple parametric null hypothesis vs $Q$
is not much better than the best \e-power
achieved for testing the composite (and usually massive) nonparametric null hypothesis.
Or, if Pitman-type notion of efficiency is to be used (as in this paper),
that the same \e-power is attained for numbers of observations
that are not wildly different.

Our use of the Gaussian model with variance 1 as assay model
motivates using \eqref{eq:E-exp}
with $S(z_1,\dots,z_n):=z_1+\dots+z_n$ as a nonparametric \e-test.
The sign and Wilcoxon versions will be natural modifications
(corresponding to relaxing the symmetry assumption,
as explained in Remark~\ref{rem:relax} below).

For all three nonparametric \e-tests considered in this paper
(Sects.~\ref{sec:ARE}--\ref{sec:Wilcoxon} below)
we will need the number $n_{\nu,2}$ of observations required by our baseline,
which is, by Lemma~\ref{lem:power},
the likelihood ratio $\dd N(\theta_\nu,1) / \dd N(0,1)$.
By \eqref{eq:optimal},
achieving an \e-power of $\beta$ requires approximately
\begin{equation}\label{eq:baseline}
  2\beta\theta_\nu^{-2}
\end{equation}
observations (namely, $\lceil 2\beta\theta_\nu^{-2}\rceil$ observations).

\begin{remark}
  In the context of regular statistical models such as Gaussian,
  it is natural to set $\theta_\nu:=c\nu^{-1/2}$.
  In this case the ``difficulty'' $\nu$
  (referred to as ``time'' in \cite[Sect.~14.3]{vanderVaart:1998})
  becomes proportional to the number of observations
  required to achieve a given \e-power.
\end{remark}

\section{Asymptotic efficiency of the Fisher-type \e-test}
\label{sec:ARE}

In the classical case, the relative efficiency of Fisher's test is 1
\cite[Chapter 7, Example 4.1]{Fraser:1957},
as first shown by Hoeffding \cite{Hoeffding:1952}
(according to Mood \cite{Mood:1954}).
Let us check that this remains true for the \e-version as well.

First we find informally a suitable \e-variable in the family \eqref{eq:E-Fisher}
and then show that it requires the optimal number \eqref{eq:baseline} of observations
to achieve an \e-power of $\beta$.
Under the symmetry model,
each observation $z_i$ is split
into its magnitude $m_i:=\left|z_i\right|$ and sign $s_i:=\sign(z_i)$.
Given the magnitudes, the signs are independent and
$
  \P(s_i=1)
  =
  1/2
$
under the null hypothesis $N(0,1)$ and
\begin{align*}
  \P(s_i=1)
  &=
  \frac
  {
    \exp
    \left(
      -\frac12(m_i-\theta_\nu)^2
    \right)
  }
  {
    \exp
    \left(
      -\frac12(m_i-\theta_\nu)^2
    \right)
    +
    \exp
    \left(
      -\frac12(-m_i-\theta_\nu)^2
    \right)
  }\\
  &=
  \frac
  {
    \exp
    \left(
      \theta_\nu m_i
    \right)
  }
  {
    \exp
    \left(
      \theta_\nu m_i
    \right)
    +
    \exp
    \left(
      -\theta_\nu m_i
    \right)
  }
\end{align*}
under the alternative hypothesis $N(\theta_\nu,1)$.
The conditional likelihood ratio for the signs is
\begin{equation*}
  \prod_{i=1}^n
  \frac
  {
    2\exp
    \left(
      \theta_\nu z_i
    \right)
  }
  {
    \exp
    \left(
      \theta_\nu m_i
    \right)
    +
    \exp
    \left(
      -\theta_\nu m_i
    \right)
  }
  =
  \prod_{i=1}^n
  \frac
  {
    \exp
    \left(
      \theta_\nu z_i
    \right)
  }
  {
    1
    +
    \theta_\nu^2 m_i^2 / 2
    +
    o(\theta_\nu^2 m_i^2)
  }.
\end{equation*}
This is Fisher's \e-test \eqref{eq:E-Fisher} corresponding to $\lambda:=\theta_\nu$.
Its observed \e-power is
\begin{equation*}
  \sum_{i=1}^n
  \left(
    \theta_\nu z_i
    -
    \theta_\nu^2 m_i^2 / 2
    +
    o(\theta_\nu^2 m_i^2)
  \right)
  =
  \theta_\nu
  \sum_{i=1}^n
  z_i
  -
  (1+o(1))
  \frac{\theta_\nu^2}{2}
  \sum_{i=1}^n
  m_i^2.
\end{equation*}
Since, under the alternative hypothesis $N(\theta_\nu,1)$,
\[
  \E
  \sum_{i=1}^n
  z_i
  =
  n \theta_\nu
\]
and
\[
  \E
  \sum_{i=1}^n
  m_i^2
  =
  \E
  \sum_{i=1}^n
  z_i^2
  =
  n
  +
  n \theta_\nu^2
  =
  (1+o(1))
  n,
\]
the \e-power is
\[
  n \theta_\nu^2
  -
  (1+o(1))
  \frac{\theta_\nu^2}{2}
  n  
  \sim
  \frac12 n \theta_\nu^2.
\]
We obtain the optimal \e-power \eqref{eq:optimal}
with $\theta=\theta_{\nu}$,
and so the asymptotic relative efficiency of Fisher's \e-test is 1.

\section{Sign \e-test}
\label{sec:sign}

In this and following sections we use \eqref{eq:E-exp}
for different statistics $S$,
and with $C$ still chosen to make $E_{\lambda}$ an admissible \e-variable.
In this section we make the simplest choice of $S(z_1,\dots,z_n)$ in \eqref{eq:E-exp},
which is the number $k$ of positive $z_i$ among $z_1,\dots,z_n$.
This gives the \emph{sign \e-test} with parameter $\lambda>0$.
The use of the signs for hypothesis testing goes back to \cite{Arbuthnott:1710}.

To obtain a useful alternative representation of the sign \e-test,
let $p\in(0,1)$ be defined by the equation
\[
  \frac{p}{1-p}
  =
  e^{\lambda}
\]
(so that $\lambda$ becomes the log-odds ratio).
The \e-variable \eqref{eq:E-exp} then becomes
\begin{equation}\label{eq:E-pre-p}
  E_{\lambda}(z_1,\dots,z_n)
  =
  e^{\lambda k - C}
  =
  p^k(1-p)^{-k}
  e^{-C}
  =
  \frac{p^k(1-p)^{n-k}}{2^{-n}}.
\end{equation}
The last expression is the likelihood ratio of an alternative to the null hypothesis,
and so is an admissible \e-variable.
This gives us the representation
\begin{equation}\label{eq:equivalent}
  E_p(z_1,\dots,z_n)
  :=
  \frac{p^k(1-p)^{n-k}}{2^{-n}}
\end{equation}
of the sign \e-test.

The equality between the last two terms in \eqref{eq:E-pre-p}
gives an explicit expression for $C$,
\[
  C
  =
  -n\log(2(1-p))
  =
  n\log\frac{1+e^{\lambda}}{2},
\]
which in turn gives the alternative representation
\begin{equation}\label{eq:E-sign-lambda}
  E_{\lambda}(z_1,\dots,z_n)
  =
  e^{\lambda k}
  \left(\frac{2}{1+e^{\lambda}}\right)^{n}
\end{equation}
of the sign \e-test.

In view of our informal alternative hypothesis,
we are often interested in $\lambda>0$, i.e., $p>1/2$.

\begin{remark}\label{rem:relax}
  Notice that in this section we are actually testing
  a wider null hypothesis than the symmetry model,
  since the magnitudes of $z_i$ do not matter.
  Namely, the sign \e-test is valid for testing the hypothesis
  that the signs of $Z_1,\dots,Z_n$ are $\pm1$ independently.
  A similar remark can also be made about the nonparametric \e-test
  discussed in the following section,
  which in fact tests an intermediate null hypothesis.
\end{remark}

As before,
we have a dependence of the sign \e-test \eqref{eq:equivalent}
on a parameter, $p$.
To get rid of this dependence,
we can, e.g., integrate \eqref{eq:equivalent} over $p\in[0,1]$, obtaining
\begin{equation*}
  E(z_1,\dots,z_n)
  :=
  2^n \mathrm{B}(k+1,n-k+1),
\end{equation*}
where $\mathrm{B}$ is the beta function.
For testing the one-sided hypothesis we can integrate \eqref{eq:equivalent}
over the uniform probability measure on $[0.5,1]$, which gives
\begin{equation*}
  E(z_1,\dots,z_n)
  :=
  2^{n+1}
  \bigl(
    \mathrm{B}(k+1,n-k+1)
    -
    \mathrm{B}(0.5;k+1,n-k+1)
  \bigr),
\end{equation*}
where the second entry of $\mathrm{B}$ stands for the incomplete beta function.

\subsection*{Efficiency of the sign test}

In this and next sections we consider the same assay parametric model
and still assume that the null hypothesis is $N(0,1)$
and the alternative is $N(\theta_\nu,1)$.
Suppose we only observe the signs $s_i$ of $z_i$,
which is sufficient when testing the null hypothesis with the sign \e-test.
By Lemma~\ref{lem:power} the largest \e-power for an \e-variable of this kind
will be achieved by the likelihood ratio for the signs.

The sign of $Z_i$ is $1$ with probability $1/2$ under the null hypothesis
and $1/2+\tilde\theta_\nu/\sqrt{2\pi}$ under the alternative
for $\tilde\theta_{\nu}\sim\theta_{\nu}$,
due to the first-order Taylor approximation
of the standard Gaussian cumulative distribution function $\Phi$.
With $k$ being the number of positive $z_i$,
the likelihood ratio for the signs is
\begin{equation*}
  \frac
  {
    \left(
      \frac12
      +
      \frac{\tilde\theta_\nu}{\sqrt{2\pi}}
    \right)^k
    \left(
      \frac12
      -
      \frac{\tilde\theta_\nu}{\sqrt{2\pi}}
    \right)^{n-k}
  }
  {
    \left(
      1/2
    \right)^{n}
  }
  =
  \left(
    1
    +
    \sqrt{\frac{2}{\pi}} \tilde\theta_\nu
  \right)^k
  \left(
    1
    -
    \sqrt{\frac{2}{\pi}} \tilde\theta_\nu
  \right)^{n-k}.
\end{equation*}
This is an instance of the sign \e-test \eqref{eq:equivalent},
corresponding to $p=1/2+\tilde\theta_\nu/\sqrt{2\pi}$.
The observed \e-power of this \e-test is
\begin{multline*}
  k
  \log
  \left(
    1
    +
    \sqrt{\frac{2}{\pi}} \tilde\theta_\nu
  \right)
  +
  (n-k)
  \log
  \left(
    1
    -
    \sqrt{\frac{2}{\pi}} \tilde\theta_\nu
  \right)\\
  =
  (2k-n)
  \sqrt{\frac{2}{\pi}} \tilde\theta_\nu
  -
  \frac{1}{\pi} n \tilde\theta_\nu^2
  +
  o(n \tilde\theta_\nu^2)
\end{multline*}
(we have used the second-order Taylor approximation).
This gives the \e-power
\begin{equation*}
  \left(
    2
    \left(
      \frac12
      +
      \frac{\tilde\theta_\nu}{\sqrt{2\pi}}
    \right)
    n
    -
    n
  \right)
  \sqrt{\frac{2}{\pi}} \tilde\theta_\nu
  -
  \frac{1}{\pi} n \tilde\theta_\nu^2
  +
  o(n \tilde\theta_\nu^2)
  =
  \frac{1}{\pi} n \tilde\theta_\nu^2
  +
  o(n \tilde\theta_\nu^2)
  \sim
  \frac{1}{\pi} n \theta_\nu^2.
\end{equation*}
To achieve an \e-power of $\beta$,
the sign \e-test needs ${} \sim \pi \beta \theta_\nu^{-2}$ observations.
Therefore, the asymptotic efficiency of the sign \e-test is $2/\pi\approx0.64$,
exactly the same as in the standard case \cite[Example 3.1]{Fraser:1957}.
(In the standard case the sign test is usually compared with the t-test,
but in this paper we use an even more basic assay parametric model;
namely, we assume that the variance is known to be 1.)

Since the asymptotic efficiency is approximately 2/3,
we can say that the sign test wastes every third observation in our Gaussian setting.
This is the least efficient of the three nonparametric \e-tests considered in this paper
when efficiency is measured using the Gaussian assay model as yardstick.

\subsection*{Sign test for Darwin's data}

It is interesting that the sign test gives
the one-sided \p-value of $0.00369$
and the two-sided \p-value of $0.00739$.
In contrast with Fisher's \p-test,
both \p-values are highly significant,
the reason being that the two negative numbers in Table~\ref{tab:Darwin}
are so large in absolute value.

\begin{figure*}[htbp]
\begin{center}
  \includegraphics[width=10cm]{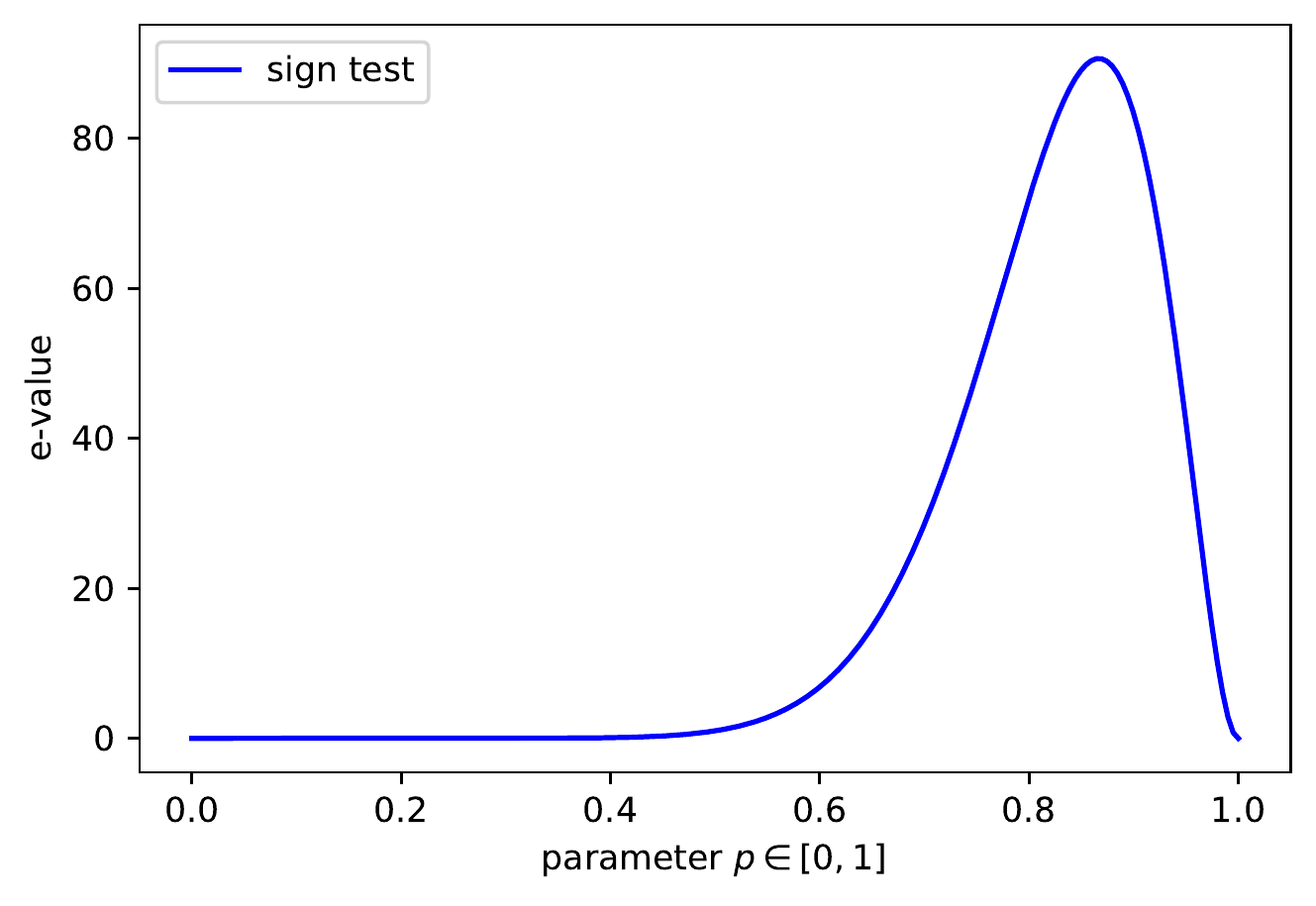}
\end{center} 
\caption{Results for the sign \e-test applied to Darwin's data}\label{fig:Darwin-sign}
\end{figure*}

Figure~\ref{fig:Darwin-sign} is an analogue of Figure~\ref{fig:Darwin}
for the sign test.
The attainable \e-values are now much larger,
and the average over all $p\in[0,1]$
is 19.310.
To use Jeffreys's \cite[Appendix~B]{Jeffreys:1961} expressions,
we have strong evidence against the null hypothesis
of cross- and self-fertilization being equally efficient.
The corresponding one-sided \e-value,
found as the average over all $p\in[0.5,1]$,
is 38.544,
and in Jeffreys's terminology it provides very strong evidence
(for cross-fertilization tending to produce taller plants,
in this context).

Table~\ref{tab:Darwin} comprises only small part
of the overwhelming evidence in favour of cross-fertilization
collected by Darwin over 11 years.
Darwin chose maize
to illustrate his and Galton's statistical methods in \cite[Chap.~1]{Darwin:1876},
but in \cite[Chaps.~2--6]{Darwin:1876} he has 99 similar tables
(with our Table~\ref{tab:Darwin} corresponding to Darwin's Table~97).
With this amount of evidence, statistics is hardly needed
to see that the evidence is really overwhelming.

\section{Wilcoxon's signed-rank \e-tests}
\label{sec:Wilcoxon}

Wilcoxon's signed-rank test \cite{Wilcoxon:1945}
is based on arranging the magnitudes $\lvert z_i\rvert$ of the observations
in the ascending order and assigning to each its \emph{rank},
which is a number in the range $\{1,\dots,n\}$:
the observation $z_i$ with the smallest $\lvert z_i\rvert$ gets rank 1,
the one with the second smallest $\lvert z_i\rvert$ gets rank 2, etc.
Notice that the symmetry model
(i.e., the uniform probability measure on \eqref{eq:b})
implies that for any set $A\subseteq\{1,\dots,n\}$,
the probability is $2^{-n}$
that the observations with the ranks in $A$ will be positive
and all other observations will be negative.
This determines the distribution (conditional on the magnitudes $\lvert z_i\rvert$)
of Wilcoxon's statistic $V_n$
defined as the sum of the ranks of the positive observations.

We will be interested in the nonparametric \e-test \eqref{eq:E-exp} with $S:=V_n$, i.e.,
\begin{equation}\label{eq:E-Wilcoxon}
  E_{\lambda}(z_1,\dots,z_n)
  :=
  \exp
  \left(
    \lambda V_n - C
  \right).
\end{equation}
The following lemma gives a convenient formula for computing $C$.

\begin{lemma}
  The value of $C$ in \eqref{eq:E-Wilcoxon} is given by
  \begin{equation}\label{eq:C}
    C
    =
    \sum_{i=1}^n
    \log
    \frac{1 + e^{\lambda i}}{2}.
  \end{equation}
\end{lemma}

\begin{proof}
  Using Fisher's conditional distribution
  (the uniform probability measure on \eqref{eq:b}),
  we can write $C$ in the form
  \[
    C
    =
    \log
    \left(
      2^{-n}
      \sum_{A\subseteq\{1,\dots,n\}}
      \exp(\lambda\Sum(A))
    \right),
  \]
  where $\Sum(A)$ is the sum of all elements of $A$.
  Setting
  \[
    \Sigma_i
    :=
    \sum_{A\subseteq\{1,\dots,i\}}
    \Lambda^{\Sum(A)},
  \]
  where $\Lambda:=\exp(\lambda)$,
  and using the recursion
  \[
    \Sigma_i = \Sigma_{i-1} + \Lambda^i \Sigma_{i-1}
  \]
  (obtained by splitting all subsets of $\{1,\dots,i\}$
  into those that do not contain $i$ and those that do),
  we obtain
  \[
    \Sigma_n
    =
    \prod_{i=1}^n
    (1+\Lambda^i).
    \qedhere
  \]
\end{proof}

Plugging \eqref{eq:C} into \eqref{eq:E-Wilcoxon},
we obtain \emph{Wilcoxon's signed-rank \e-test}
\begin{equation}\label{eq:E-Wilcoxon-full}
  E_{\lambda}(z_1,\dots,z_n)
  :=
  \exp(\lambda V_n)
  \prod_{i=1}^n\frac{2}{1+e^{\lambda i}}.
\end{equation}

\subsection*{Efficiency of Wilcoxon's signed-rank \e-test}

Our derivation in this subsection will follow \cite[Example 3.3.6]{Lehmann:1999}.
The statistic
\begin{equation}\label{eq:T_n}
  T_n
  :=
  V_n / \binom{n}{2},
\end{equation}
$V_n$ being Wilcoxon's signed-rank statistic
defined at the beginning of this section,
is asymptotically normal both under the null hypothesis $N(0,1)$,
\begin{equation}\label{eq:T-null}
  T_n
  \sim
  N
  \left(
    \frac12,
    \frac{1}{3n}
  \right),
\end{equation}
and under the alternative hypothesis $N(\theta_\nu,1)$,
\begin{equation}\label{eq:T-alternative}
  T_n
  \sim
  N
  \left(
    \frac12+\frac{\theta_\nu}{\sqrt{\pi}},
    \frac{1}{3n}
  \right).
\end{equation}
The mean value $1/2+\theta_\nu/\sqrt{\pi}$ in \eqref{eq:T-alternative}
is found as the first-order approximation to the probability of $Z_1 + Z_2 > 0$,
where $Z_1$ and $Z_2$ are independent and distributed
according to the alternative hypothesis $N(\theta_\nu,1)$
(see \cite[(3.3.40)]{Lehmann:1999}).
Namely, it is obtained from $Z_1+Z_2\sim N(2\theta_\nu,2)$
and from the standard Gaussian density being $1/\sqrt{2\pi}$ at 0.

From \eqref{eq:T-null} and \eqref{eq:T-alternative}
we obtain the asymptotic likelihood ratio
\begin{equation}\label{eq:LR-Wilcoxon}
  \frac
  {
    \exp
    \left(
      -\frac12
      \left(T_n-\frac12-\frac{\theta_\nu}{\sqrt{\pi}}\right)^2
      /
      \frac{1}{3n}
    \right)
  }
  {
    \exp
    \left(
      -\frac12
      \left(T_n-\frac12\right)^2
      /
      \frac{1}{3n}
    \right)
  }
  =
  \exp
  \left(
    3 n
    \left(T_n-\frac12\right)
    \frac{\theta_\nu}{\sqrt{\pi}}
    -
    \frac{3n}{2}
    \frac{\theta_\nu^2}{\pi}
  \right)
\end{equation}
(of the form \eqref{eq:E-Wilcoxon}; see below).
The observed \e-power is obtained by removing the $\exp$,
and then the \e-power is obtained by taking the expectation
w.r.\ to $T_n$ distributed as \eqref{eq:T-alternative}.
Therefore, the \e-power is, asymptotically,
\[
  3 n
  \frac{\theta_\nu}{\sqrt{\pi}}
  \frac{\theta_\nu}{\sqrt{\pi}}
  -
  \frac{3n}{2}
  \frac{\theta_\nu^2}{\pi}
  =
  \frac{3n}{2}
  \frac{\theta_\nu^2}{\pi}.
\]
The number of observations required for achieving an \e-power of $\beta$
is, asymptotically,
\[
  \frac{2\pi}{3}
  \beta
  \theta_\nu^{-2}.
\]
Comparing this with the baseline \eqref{eq:baseline}
gives the asymptotic relative efficiency of $3/\pi\approx0.955$,
as in the classical case.
Wilcoxon's test wastes one observation out of about 22
(under the Gaussian model as compared with the \e-test optimized for that model).

The approximate \e-test used in this calculation
(given by the right-hand side of \eqref{eq:LR-Wilcoxon})
is of the form \eqref{eq:E-Wilcoxon} with
\[
  \lambda
  :=
  \frac{3n\theta_\nu}{\binom{n}{2}\sqrt{\pi}}
\]
(obtained by expressing \eqref{eq:LR-Wilcoxon} in terms of $V_n$ using \eqref{eq:T_n}).
This, however, ignores the definition of $C$ in \eqref{eq:E-Wilcoxon}.
In practical application we should use, of course,
the precise expression \eqref{eq:E-Wilcoxon-full}.

\section{Directions of further research}
\label{sec:conclusion}

In the previous sections we mentioned several limitations of our definitions.
In this concluding section we will add further details.

\subsection*{The notion of \e-power as used in the definition of efficiency}

Our notion of \e-power for an \e-variable $E$ is crude in that it depends only
on the expectation of $\log E$, as explained in Remark~\ref{rem:full}.
This crudeness is inherited by our definition of the asymptotic relative efficiency of \e-tests.
According to our definition in Sect.~\ref{sec:power},
the asymptotic relative efficiency is $c$ if $n_{\nu,2} \sim c n_{\nu,1}$.
This statement will be particularly useful
if, under the alternative hypothesis,
the full distribution of the original likelihood ratio,
such as \eqref{eq:LR} for $\theta=\theta_{\nu}$ and $n_{\nu,2}$ observations,
is close, in a suitable sense, to the full distribution of the \e-test,
such as \eqref{eq:E-Fisher}, \eqref{eq:E-sign-lambda}, or \eqref{eq:E-Wilcoxon-full}
(with $n_{\nu,1}$ observations and the corresponding value of the parameter).
Therefore, a fuller treatment of asymptotic relative efficiency
will not use \e-power directly (which will make it more complicated).

\subsection*{Definition of efficiency in terms of mixtures}

Our definition of Pitman-type efficiency is close to being a direct translation
of the classical one.
It considers the alternatives $N(0,\theta_\nu)$ that approach the null hypothesis $N(0,1)$
as the difficulty $\nu$ increases.
In the classical case,
this works perfectly for many popular assay models
because of the existence of a uniformly most powerful test:
the optimal size $\alpha$ critical region does not depend on $\nu$
(assuming $\theta_\nu>0$).
In the \e-case, on the contrary,
the optimal \e-variable does depend on $\nu$.

A possible alternative definition would be to replace $N(\theta_{\nu},1)$
by a mixture $\int N(\theta,1) \mu_\nu(\dd\theta)$ of $N(\theta,1)$
w.r.\ to a probability measure $\mu_\nu(\dd\theta)$
that is increasingly concentrated around $\theta=0$ as $\nu\to\infty$.
In a sense, the assay statistical model considered in this paper is ``pure''
in that it consists of pure Gaussian distributions.
Considering mixtures $\int N(\theta,1) \mu_\nu(\dd\theta)$ would make the results more realistic
but would also make the definitions more complicated.

\subsection*{Other assay models}

In our efficiency results,
the Gaussian model can be replaced by other statistical models.
It is particularly interesting to compare nonparametric \e-tests
with the optimal \e-tests under those models;
nowadays, comparison with the t-test,
which was done in many of the classical papers
(e.g., \cite{Hodges/Lehmann:1956}),
looks less convincing for non-Gaussian assay models.

Our choice of the form \eqref{eq:E-exp} of the nonparametric \e-tests
considered in this paper was motivated by the Gaussian assay model:
see the likelihood ratio \eqref{eq:LR}.
Using other assay models would lead to other nonparametric \e-tests.
Therefore, varying the assay model
may be a useful design tool for nonparametric \e-tests.

\subsection*{Other notions of efficiency}

The Pitman-type notion of efficiency is ``local'',
in the sense of being defined in terms of progressively more difficult alternatives
that tend to the null hypothesis as $\nu\to\infty$.
It is the most popular notion of efficiency for nonparametric tests,
but it would be interesting to develop \e-versions
of other, non-local, notions of asymptotic relative efficiency
(see, e.g., \cite[Chap.~1]{Nikitin:1995}).

\section*{Acknowledgements}

Many thanks to the anonymous reviewers of the journal version of this paper.
Vladimir Vovk's research has been supported by Mitie.
Ruodu Wang acknowledges financial support
by grants CRC-2022-00141 and RGPIN-2018-03823
from the Natural Sciences and Engineering Research Council of Canada.

\end{document}